\documentclass[leqno, 10pt]{amsart}
\usepackage{enumerate,tikz-cd, amscd, mathtools, float, hyperref}
\usepackage{extpfeil}
\hypersetup{colorlinks,linkcolor={blue},citecolor={blue},urlcolor={cyan}}

\newtheorem{theorem}{Theorem}[section]
\newtheorem{conjecture}[theorem]{Conjecture}
\newtheorem{proposition}[theorem]{Proposition}
\newtheorem{lemma}[theorem]{Lemma}
\newtheorem{definition}[theorem]{Definition}
\newtheorem{corollary}[theorem]{Corollary}

\newtheorem{example}[theorem]{Example}

\newtheorem{remark}[theorem]{Remark}

\DeclareMathOperator{\Hom}{Hom}

\DeclareMathOperator{\End}{End}

\DeclareMathOperator{\im}{Im}
\DeclareMathOperator{\rdim}{rdim}

\DeclareMathOperator{\Ind}{Ind}

\newcommand\floor[1]{\lfloor #1 \rfloor}
\newcommand{\newterm}{\textsf}

\title{Rouquier dimension is Krull dimension for normal toric varieties}

\author[Favero]{David Favero}
\address{
	\begin{tabular}{l}
		David Favero \\
		\hspace{.1in} University of Minnesota, School of Mathematics \\
		\hspace{.1in} 127 Vincent Hall, 206 Church Street, Mpls., MN 55455 \\
		\hspace{.1in} Korea Institute for Advanced Study \\
		\hspace{.1in} 85 Hoegiro, Dongdaemun-gu, Seoul, Republic of Korea 02455 \\
		\hspace{.1in} Email: {\bf favero@ualberta.ca} \\
	\end{tabular}
}
\author[Huang]{Jesse Huang}
\address{
	\begin{tabular}{l}
		Jesse Huang \\
		\hspace{.1in} University of Alberta, Department of Mathematical and Statistical Sciences \\
		\hspace{.1in} Central Academic Building 632, Edmonton, AB, Canada T6G 2C7 \\
		\hspace{.1in} Email: {\bf jesse.huang@ualberta.ca} \\
	\end{tabular}
}

\begin{document}

\maketitle

\begin{abstract}
    We prove that for any normal toric variety, the Rouquier dimension of its bounded derived category of coherent sheaves is equal to its Krull dimension. Our proof uses the coherent-constructible correspondence to translate the problem into the study of Rouquier dimension for certain categories of constructible sheaves.
\end{abstract}

\section{Introduction}

The Rouquier dimension of a triangulated category $\mathcal T$ is a measure of its homological complexity.  More precisely, given a generator $G \in \mathcal T$, the \newterm{generation time} is the minimal number of exact triangles needed to form every object of $\mathcal T$ from $G$.  The Rouquier dimension is simply the infimum of these generation times.

Rouquier, who introduced this invariant in \cite{Rou08}, primarily studied it for $\mathcal T = Coh(X)$, the derived category of coherent sheaves of a scheme $X$.  Notably, he showed that the Rouquier dimension of $Coh(X)$ is bounded below by the Krull dimension of $X$ and remarked that he knew of no cases where the two were not equal.  Orlov \cite{Orl09} went on to prove that the two invariants agree in dimension 1 and proposed that Rouquier dimension and Krull dimension should agree for all smooth algebraic varieties.
\begin{conjecture}[Orlov] \label{conj: Orlov}
For any smooth algebraic variety, the Rouquier dimension of its derived category of bounded coherent sheaves agrees with the Krull dimension of the variety.
\end{conjecture}

While this conjecture is still largely open, is has been established in various special cases \cite{Rou08, Orl09, BF12, BFK12, BFK14, Yan16, BFK19, BDM19, Pir19, Ola21, EXZ21, BC21, BS22}.\footnote{We are not sure if this list is exhaustive.}  In this note, we settle the toric case:
\begin{theorem}[=Corollary \ref{cor: rdim=dim variety}] \label{thm: main theorem intro}
For any normal toric variety Conjecture \ref{conj: Orlov} holds.
\end{theorem}
This result has also been obtained independently and will appear in forthcoming work of Hanlon-Hicks-Lazarev \cite{HHL-b}.

Our method of proof passes through a form of homological mirror symmetry for toric varieties (and certain toric stacks) called the coherent-constructible correspondence.  This correspondence goes back to the pioneering works of Fan-Liu-Treumann-Zaslow \cite{FLTZ11, FLTZ12, FLTZ14}, however, we use the most general form  due to Kuwagaki \cite{Kuw20}.  This provides an equivalence
\[
\kappa: Coh(X_\Sigma) \cong Sh^{w}(\Lambda_\Sigma)
 \]
between the derived categories of coherent sheaves on the toric variety $X_\Sigma$ and a certain class of constructible sheaves on a torus. 

Hence, to obtain the theorem, it is enough to study Rouquier dimension for certain categories of constructible sheaves.  It turns out that such categories have a natural generator $\mathcal P^\Lambda$ which we call the \newterm{probe generator}.  Therefore, we study the generation time of this object, providing upper and lower bounds which in the case of a torus are sharp (yielding our main result as a consequence).
\begin{theorem}[=Example~\ref{ex: torus}]
Let $\mathbb T = (S^1)^n$ be a real torus and $\Lambda$ be a Lagrangian skeleton.  Then the generation time of the probe generator is $n$.     
\end{theorem}

For completeness, we also track the probe generator through the equivalence $\kappa^{-1}$.  As it turns out, the summands of the probe generator correspond to a natural stratification of the torus first considered by Bondal-Ruan \cite{Bon06} who established a bijection between strata and summands of the toric Frobenius pushforward of $\mathcal O$.  Therefore, unsurprisingly, we find that the probe generator is taken to the toric Frobenius pushforward of $\mathcal O$ (the same generator studied in \cite{BDM19}).

\section*{Acknowledgment}
We thank Andrew Hanlon for notifying us about \cite{HHL-b} and providing us with an early draft with an independent proof of Theorem \ref{thm: main theorem intro} found therein. We are also grateful to Evgeny Shinder and Martin Kalck for providing us with the reference \cite{Kaw19}.  We thank Matt Ballard and Alex Duncan for their excellent explanations of the results in \cite{BDM} and additional discussions at the Banff International Research Station. J. Huang is supported by a Pacific Institute for the Mathematical Sciences Postdoctoral Fellowship and by NSERC through the Discovery Grant program. D. Favero is supported by NSERC through the Discovery Grant and Canada Research Chair programs.

\section*{Notation and conventions}
\begin{itemize}
    \item All categories are dg derived. 
    \item We use the notation in e.g. \cite{Nad16} for various sheaf categories with certain stratification or singular support condition. $Sh^\diamond$ means the dg derived category of constructible sheaves of $\mathbb{C}$-modules (of possibly infinite rank), $Sh^w$ means compact objects in $Sh^\diamond$, and $Sh^c$ means the dg derived category of constructible sheaves of $\mathbb{C}$-modules of finite rank.
    \item The generation time of a generator $G$ in a triangulated category is denoted by $t(G)$. The Loewy length of an algebra $A$ is denoted by $LL(A)$.
    
\end{itemize}

\section{Generation time for probe generators}
Let $X$ be a finite regular CW complex satisfying the \newterm{axiom of the frontier}
$$
\overline{e_\alpha}\cap e_\beta \neq \emptyset \Leftrightarrow e_\beta \subset \overline{e_\alpha} \text{ for cells } e_\beta \text{ and } e_\alpha,
$$
equipped with its natural cell poset stratification $S_{CW}$ defined by the incidence relation above. The incidence algebra, $A$, of this poset is an HPA \cite[Definition 3.1]{FH22} associated to the quiver $Q$:
\begin{align*}
Q_0 & := \{e_i \in S_{CW} \} \\
\# Q_1 (e_j \rightarrow e_i) & = \begin{cases} 1 & \text{ if } \overline{e_i} \cap e_j \neq \emptyset \text{ and } j=i-1\\
0 & \text{ otherwise}
\end{cases}
\end{align*}
with ideal generated by path homotopy relation among exit paths. One can write $A = \End (G_A)$, with $G_A=\bigoplus_{S_\alpha\in S_{CW}} \mathcal P_\alpha$ and $\mathcal P_\alpha$ are the exit sheaves from strata $S_\alpha$ \cite[Proposition 4.16]{FH22}. 

\begin{lemma} \label{Koszul}
The algebra $A$ is Koszul. The direct sum of the left dual exceptional collection is
    $$G_A^!:=\bigoplus_{\alpha \in S_{CW}} i_{\alpha!} \mathbb{C}_{S_\alpha}[k_\alpha],$$
where $k_\alpha$ is the position $\mathcal P_\alpha$ appears in the exceptional collection. The endomorphism algebra of $G_A^!$ in $Sh^c_{S_{CW}}(X)$ is the Koszul dual $A^!$.
\end{lemma}
\begin{proof}
The order complex of any open interval $(e_\alpha, e_\beta)$ in the cell poset, which is the face link of the minimal cell $e_\alpha$ inside the maximal cell $e_\beta$ bounding the interval, is always homeomorphic to a sphere of dim $\dim e_\beta-\dim e_\alpha -2$, hence the cell poset is locally Cohen-Macaulay. By \cite[Proposition, Section 1.6]{Pol95}, $A$ is Koszul. (See also \cite{Yan04, Yan05}.) 

Now observe that
    $$ \Hom (\mathcal P_\beta, i_{\alpha!} \mathbb{C}_{S_\alpha}) = \begin{cases} \mathbb{C} & \alpha=\beta \\ 0 & \text{else}\end{cases}$$
    Hence $i_{\alpha!} \mathbb{C}_{S_\alpha}[k]$ is the left dual of $\mathcal P_{\alpha}$.
\end{proof}

\begin{remark}
Explicitly, we can define an ideal $I \subseteq kQ$ generated by
\[
I := \langle a(e_i, e_{i-1})a(e_{i-1}, e_{i-2}) - a(e_i, e'_{i-1})a(e'_{i-1}, e_{i-2}) \rangle.
\]
Then $A = kQ/I$. Our assumptions on $X$ imply, by \cite[Theorem 4.2]{CF67}, that for any codimension 2 indicence relation $\sigma<_2 \sigma'$, there are exactly two intermediate $\tau$, such that $\sigma<_1 \tau<_1\sigma'$. Hence the perpendicular of $I$ in $kQ^{\text{op}}$ is given by
\[
I^\perp := \langle  a(e_i, e_{i-1})a(e_{i-1}, e_{i-2}) + a(e_i, e'_{i-1})a(e'_{i-1}, e_{i-2}) \rangle,
\]
which gives the quadratic dual algebra
\[
A^! := kQ^{\text{op}}/I^\perp.
\]
\end{remark}
%We have $A^!\simeq A^{op}\simeq \End(G_A^!)$.
\begin{remark}
In this simple case, $Sh^w_{S_{CW}}(X)= Sh^c_{S_{CW}}(X)$.
\end{remark}
\begin{remark}
The $A_\infty$-structure on both dg endomorphism algebras can be seen to be formal. Namely, an exit sheaf $\mathcal P_\alpha \in Sh_{S_{CW}}^c(X)$ is a complex of sheaves concentrated in degree 0. Since $\mathcal P_\alpha$ represents the stalk functor at $p\in S_\alpha$ \cite{FH22}, we have $\Hom_{Sh^c_{S_{CW}}(X)}(\mathcal P_\alpha, \mathcal P_{\alpha'})= \mathcal P_{\alpha'}|_p$, which is in degree 0. Hence all morphisms among the exit sheaves are in degree 0, which implies $\End_{Sh^c_{S_{CW}}(X)}(G_A)$ is formal. The Koszul property also forces the $A_\infty$-structure of $\End_{Sh^c_{S_{CW}}(X)}(G_A^!)$ to be formal for degree reasons.
\end{remark}

\begin{proposition} \label{prop: CW gen time}
Given $(X, S_{CW})$, we have
\[
t(G_A) = t(G_A^!) = \dim X
\]
\end{proposition}
\begin{proof}
By Lemma~\ref{Koszul}, $A$ is Koszul, and the quadratic dual $\End(G^!_A) = A^!$.  Hence
by \cite[Proposition 3.26]{BFK12} we have $t(G_A) = LL(A^!) - 1 = \dim X$.  Symmetrically, $\End(G_A) = A$ and $t(G^!_A) = LL(A) - 1 = \dim X$.
\end{proof}

% \begin{lemma}\label{upper bound}
%     $$\rdim DSh_{S_{CW}} (X) \leq \dim X.$$
% \end{lemma}
% \begin{proof}
%     A sequence of entrance paths has length at most $\dim X$. Hence, $LL(A_{S_{CW}})\leq \dim X$. By \cite[Proposition 3.26]{BFK12}, the generation time of the dual collection, which is $G_A^!$ by Lemma~\ref{dual collection} is at most $LL_\infty(A_{S_{CW}})$. By Lemma~\ref{formal}, $LL_\infty(A_{S_{CW}})=LL(A_{S_{CW}})$, hence follows the lemma.
% \end{proof}

We now give a (somewhat restrictive) definition of Lagrangian skeleta for convenience.

\begin{definition} \label{def: skeleton}
    A \newterm{Lagrangian skeleton} on a smooth manifold $X$ is a conical Lagrangian in $T^*X$ containing the zero section, embedded as a closed subset into the union of conormals of all strata in some appropriate Whitney stratification of $X$.
    
    If $\mathcal S$ is a coarsening of such stratification on $X$ such that the induced map on strata is a poset map (where the partial ordering is given by $S_i\cap \overline S_j \neq \emptyset$), we call the Lagrangian skeleton
    $$\Lambda_S:=\bigcup_{S_\alpha \in \mathcal S} ss(i_{\alpha!} \mathbb{C}_{S_\alpha}).$$
    the \newterm{stratification skeleton} of $\mathcal S$.
\end{definition}

Now suppose $X$ is a smooth manifold and $\Lambda'\subset T^*X$ is a Lagrangian skeleton. Recall from \cite{Nad16} that if $\Lambda\subset \Lambda'$ is a closed subskeleton, there is a localization functor
$$
\rho^L: Sh^\diamond_{\Lambda'}(X)\rightarrow Sh^\diamond_{\Lambda}(X),
$$
defined as the left adjoint to the inclusion $\rho: Sh^\diamond_{\Lambda}(X) \rightarrow Sh^\diamond_{\Lambda'}(X)$. Moreover, $\rho^L$ preserves compact objects. Restriction to compact objects gives the stop removal functor 
$$
\rho^{L,w}: Sh^w_{\Lambda'}(X)\rightarrow Sh^w_{\Lambda}(X).
$$

% When $Sh^w_{\Lambda'}(X)$ is smooth and proper (need more on this), we have $Sh^w_{\Lambda'}(X)=Sh^c_{\Lambda'}(X)$, and we have an adjoint pair $(i^c, i^{L,c})$. In this case, we say $\Lambda$ is an \newterm{admissible} subskeleton.

% \begin{definition}
% Let $\Lambda \subseteq \Lambda'$ be a subskeleton.  We say that $\Lambda$ is \newterm{admissible} if the natural inclusion respects constructible objects \jh{the map is backward. see above}

% \[
% Sh^c_\Lambda \to Sh^c_{\Lambda'} 
% \]
% \end{definition}

\begin{definition}
Let $\Lambda \subseteq T^*X$ be a Lagrangian skeleton.  We call the representative of the stalk functor at $v$, that is, $\mathcal P^\Lambda_v \in Sh^w_\Lambda(X)$ satisfying
\[
\Hom_{Sh^\diamond_\Lambda(X)}(\mathcal P_v, F) = F_v,\  \forall F\in Sh^\diamond_\Lambda(X)
\]
the \newterm{stalk probe} at $v$.
\end{definition}

\begin{remark}
 Here, we do not require $v$ to be inside the smooth locus of $\Lambda$. This is allowed since taking the stalk at any point $v\in X$ is a cocontinuous functor $Sh^\diamond_\Lambda\rightarrow \mathbb{C}\text{-Mod}$.
\end{remark}

\begin{lemma} \label{lem: probe to probe}
Let $\Lambda \subseteq \Lambda'$ be a closed subskeleton, then
\[
\rho^{L,w}(\mathcal P_v^{\Lambda'}) = \mathcal P_v^\Lambda.
\]
\end{lemma}
\begin{proof}
Let $\mathcal F\in Sh^\diamond_{\Lambda}(X)$. We have
\begin{align*}
\Hom_{\Lambda}(\rho^{L,w}(\mathcal P_v^{\Lambda'}), \mathcal F) & = \Hom_{\Lambda'}(\mathcal P_v^{\Lambda'}, \mathcal F) & \text{ by adjunction}\\
& = \mathcal F_v & \text{ by definition} 
\end{align*}
\end{proof}

Let $\mathbb{L} \subset T^*X$ be a Lagrangian skeleton. The collection of stalk probes at all $x\in X$ induce the following stratification on $X$:
$$p \text{ and } q \text{ are on the same stratum} \Leftrightarrow \mathcal P_p^\mathbb L=\mathcal P_q^\mathbb L.$$ 
\begin{definition}
    The above stratification is said to be the \newterm{$\mathbb L$-probe stratification}, denoted by $S_\mathbb L$.
\end{definition}

\begin{lemma}\label{lem: probe stratification}
For any closed subskeleton $\mathbb L$ of a block stratification skeleton $\Lambda_S$, the $\mathbb L$-probe stratification $S_\mathbb{L}$ is a coarsening of $S$.
\end{lemma}
\begin{proof}
The stalk probes are exit sheaves whenever $\Lambda=\bigcup_{S_\alpha\in S} ss(i_{\alpha!}\mathbb{C}_{S_\alpha})$ for a block stratification $S$ in the sense of \cite{FH22}.
\begin{align*}
p\sim_S q & \Leftrightarrow \mathcal P^{\Lambda_S}_p=\mathcal P^{\Lambda_S}_q & \text{since exit sheaves only depend on the stratum} \\
& \Rightarrow \rho ^{L,w} \mathcal  P^{\Lambda_S}_p =  \rho^{L,w} \mathcal P^{\Lambda_S}_q  & \text{ trivially}\\
& \Rightarrow \mathcal P^{\mathbb L}_p=\mathcal P^{\mathbb L}_q & \text{by Lemma \ref{lem: probe to probe}}\\
& \Leftrightarrow p\sim_{S_\mathbb L} q & \text{by definition}
\end{align*}
\end{proof}

\begin{definition}\label{def:probe generator}
Suppose $S_{CW}$ is a regular CW complex stratification on a smooth manifold $X$ with stratification skeleton $\Lambda_{CW}$ (see Definition \ref{def: skeleton}), and let $\Lambda \subseteq \Lambda_{CW}$ be a closed subskeleton. For each cell $e_\alpha$, choose a base point $e_{\alpha, v}$. The \newterm{probe generator} is the direct sum of stalk probes
\[
\mathcal P^\Lambda := \bigoplus_{S_\alpha\in S_{CW}} \mathcal P^\Lambda_{e_{\alpha,v}}
\]
considered as an object of $Sh^w_{\Lambda}$.
\end{definition}

\begin{lemma} \label{lem: subskeleton time}
Let $\Lambda \subseteq \Lambda'$ be a closed subskeleton of a Lagrangian skeleton $\Lambda\subset T^*X$, then 
\[
t(\mathcal P^{\Lambda}) \leq t(\mathcal P^{\Lambda'})
\]
\end{lemma}
\begin{proof}
The functor $\rho^{L,w}$ has dense image by \cite[Corollary 4.22]{GPS}.  Hence
\begin{align*}
t(\mathcal P^{\Lambda'}) & \geq t(\rho^{L,w}(\mathcal P^{\Lambda'})) & \text{since $\rho^{L,w}$ has dense image} \\
& = t(\mathcal P^\Lambda) & \text{ by Lemma~\ref{lem: probe to probe}.}
\end{align*}
\end{proof}

% \begin{corollary} \label{upper bound block}
%     If $S$ is a block stratification on $X$ coarsened from $S_{CW}$, then $\rdim DSh_S(X) \leq \dim X$.
% \end{corollary}
% \begin{proof}
% Let $f: S_{CW} \rightarrow S$ denote the poset map of strata sending each cell stratum to the unique stratum in $S$ containing it. The essential image of
% \begin{align*}
% F: DSh_{S_{CW}}(X) &  \rightarrow DSh_S(X) \\
%  i_{\alpha!} \mathbb{C}_{S_\alpha} & \mapsto i_{f(S_\alpha)!} \mathbb{C}_{f(S_\alpha)}
% \end{align*}
% contains $i_{S_\beta !} \mathbb{C}_{S_\beta}$ for all $S_\beta\in S$ since $f$ is surjective. Since $S$ is a block stratification, $i_{S_\beta !} \mathbb{C}_{S_\beta}$ generates by \cite[Proposition 4.6]{FH22}. Hence $F$ is essentially surjective. This implies $t(F(G_A^!)) \leq t(G_A^!)$. Now 
% \begin{align*}
% \rdim DSh_S(X) & \leq t(F(G_A^!))\leq t(G_A^!) & \\
%  & \leq \dim X & \text{by Lemma~\ref{upper bound}.}
% \end{align*}
% \end{proof}

\begin{theorem} \label{thm: main theorem}
Let $X$, $\Lambda_{CW}$, and $\Lambda\subset \Lambda_{CW}$ be as in Definition~\ref{def:probe generator}. Then,
\[
LL(H^*(X,k)) - 1 \leq t(\mathcal P^\Lambda) \leq \dim X
\]
In particular, the Rouquier dimension of $Sh^w_\Lambda(X)$ is at most $\dim X$.
\end{theorem}
\begin{proof}
Let $c \in H^i(X,\mathbb C)$ be a cohomology class.  This induces a map of constructible sheaves
\[
c: \mathbb C_X \to \mathbb C_X[i].
\]
For $i>0$, the map $c$ is $\mathcal P^\Lambda$-ghost as
\[
\Hom(\mathcal P^\Lambda, \mathbb C_X[i]) = \bigoplus_{e_v} \mathbb C[i]
\]
is concentrated in degree $i$.  Hence any nontrivial cup product $c_1 \cup ... \cup c_n$
gives a nontrivial $\mathcal P^\Lambda$-ghost sequence.  The ghost lemma (see e.g. \cite[Lemma 4.11]{Rou08} or \cite[Lemma 2.12]{BFK12}) gives the lower bound.

For the upper bound, we have
\begin{align*}
t(\mathcal P^{\Lambda}) & \leq t(\mathcal P^{\Lambda_{CW}}) & \text{ by Lemma~\ref{lem: subskeleton time}} \\
& = \dim X & \text{by Proposition~\ref{prop: CW gen time}}.
\end{align*}
\end{proof}

\begin{example} \label{ex: torus}
If $X$ is a torus then $H^*(X,k)$ is isomorphic to the exterior algebra $\Lambda^\bullet k^{\dim X}$.  Hence $LL(H^*(X,k)) = \dim X +1$ and $t(\mathcal P) = \dim X$. 
\end{example}

\section{Application to toric varieties}
We now consider the Langragian skeleta appearing as mirrors of toric varieties.  For this, let $U \subseteq \mathbb A^n$ be an open toric subset of $\mathbb A^n$ defined by a subfan $\widetilde \Sigma \subseteq \Sigma_{\mathbb A^n}$.  Let $G$ be any subgroup of $\mathbb G_m^n$, the maximal torus of $\mathbb A^n$. We define $\widetilde M := \Hom(\mathbb G_m^n, \mathbb G_m)$ and $\widehat G:= \Hom(G, \mathbb G_m)$.  Since $\Hom(-, \mathbb G_m)$ is exact, applying it to the inclusion we get an exact sequence
\[
0 \to M \xrightarrow{\beta^\vee} \widetilde M \xrightarrow{\mu} \widehat G \rightarrow 0
\]
where $M$ is defined as the kernel of $\mu$.

Taking $\Hom(-, \mathbb Z)$ gives a four-term exact sequence
\[
0\rightarrow \widehat G^\vee \xrightarrow{\mu^\vee} \widetilde{N} \xrightarrow{\beta} N \rightarrow \text{coker } \beta \rightarrow 0.
\]
We adopt the setup in \cite{Kuw20}, asking that $\beta_\mathbb{R}$ induces a combinatorial isomorphism of fans. That is, we assume
\begin{enumerate}
    \item The restricted map $\beta_{\mathbb R}|_{\widetilde \sigma}: \widetilde \sigma \xrightarrow{\sim} \sigma$ is an isomorphism of cones for all $\widetilde \sigma\in \widetilde \Sigma$.
    \item There is a poset isomorphism
    $$\widetilde \Sigma \simeq \Sigma$$
    where $\Sigma$ is a fan consisting of all cones $\sigma=\beta_\mathbb{R}(\widetilde{\sigma})$.
    
\end{enumerate}

The above assumption implies $\Sigma$ is simplicial. The quotient stack $[U/G]$ is a smooth Deligne-Mumford toric stack (see \cite{GS15}) whose coarse moduli space is the toric variety $X_\Sigma$. We note that if $X_\Sigma$ is smooth, then $X_\Sigma$ is isomorphic to $[U/G]$.

Furthermore, since $\widetilde \Sigma$ is a subfan of the standard fan for $\mathbb A^n$, all cones $\widetilde \sigma\in \widetilde \Sigma$ are of the form $\widetilde \sigma_I= \mathbb R_{\geq 0}^I\subset \widetilde N_\mathbb{R}$, with dual cone $\widetilde \sigma_I^\vee=\mathbb R^{I}_{\geq 0} \times \mathbb R^{I^c}$.
Let 
\begin{align*}
\widetilde \Lambda & := \bigcup_{\widetilde \sigma_I \in \widetilde \Sigma, \widetilde D\in \widetilde M} (\widetilde \sigma_I^\perp +\widetilde D) \times (-\widetilde \sigma_I)\\
& = \bigcup_{\widetilde \sigma_I \in \widetilde \Sigma, \widetilde D\in \widetilde M} (\mathbb R^{I^c} +\widetilde D)\times \mathbb R_{\leq 0}^{I}\\
% & = \bigcup_{\widetilde \sigma_I \in \widetilde \Sigma, \widetilde D\in \widetilde M} ss(\widetilde j_{I,\widetilde D!} \mathbb C_{\widetilde D+\mathbb{R}^{I^c}_+ \times \mathbb R^I})\\
% & = \bigcup_{\widetilde \sigma_I \in \widetilde \Sigma, \widetilde D\in \widetilde M} ss(\widetilde j_{I,\widetilde D!} \mathbb C_{\widetilde D+\Int \widetilde \sigma_I^\vee})\subset T^*\widetilde M_\mathbb{R}
\end{align*}
be the $\widetilde{M}$-periodic FLTZ skeleton associated to the fan $\widetilde \Sigma$. Let $\pi:\widetilde M_\mathbb{R} \rightarrow \widetilde M_\mathbb R/M$ be the projection, and define
% where $\widetilde j_{\widetilde D,I}: \widetilde D+(\mathbb{R}^{I^c}_+ \times \mathbb R^I) \hookrightarrow \widetilde M_\mathbb R$ is the inclusion. 

\[
\Lambda_{\widetilde \Sigma} := \widetilde  \Lambda/M.
\]
Then, the fiber skeleton
$$\Lambda_\Sigma:=\Lambda_{\widetilde \Sigma}|_{M_\mathbb{R}/M}$$
is the (non-equivariant) mirror skeleton to $[U/G]$.

\begin{lemma}\label{lem: probe mod M}
    $$\mathcal P^{\Lambda_{\widetilde \Sigma}}_{\pi(p)} =\pi_! \mathcal P^{\widetilde \Lambda}_p. $$
\end{lemma}
\begin{proof}
    \begin{align*}
        \Hom(\pi_! \mathcal P^{\widetilde \Lambda}_p, F) & = \Hom( \mathcal P^{\widetilde \Lambda}_p, \pi^! F) & \text{ by adjunction}\\
        & = \Hom( \mathcal P^{\widetilde \Lambda}_p, \pi^* F) & \text{ since $\pi$ is a covering map}\\
        & = (\pi^* F)_p & \text{ by definition of stalk probe}\\
        & = F_{\pi(p)} & 
    \end{align*}
\end{proof}

Since $\pi_! \mathcal{P}_{\widetilde D}^{\widetilde \Lambda}$ does not depend on the choice of lift $\widetilde{D}$, we write $\mathcal P_D^{\Lambda_{\widetilde \Sigma}}:=\pi_! \mathcal{P}_{\widetilde D}^{\widetilde \Lambda}$.

We will use the following noncharacteristic deformation lemma.

\begin{lemma}[\cite{She22}, Theorem 14]\label{lem: slice equivalence}
Let $i_b: \mu_{\mathbb R}^{-1}(b)/M \hookrightarrow \widetilde{M}_\mathbb{R}/M$. The pullback $$i_b^*: Sh^\diamond_{\Lambda_{\widetilde \Sigma}}\rightarrow Sh^\diamond_{\Lambda_{\widetilde \Sigma}|_{\mu_{\mathbb R}^{-1}(b)}}$$
is an equivalence of categories.
\end{lemma}

In particular, taking $b=0$ gives an equivalence of categories
$$i^*: Sh^\diamond_{\Lambda_{\widetilde \Sigma}}\xrightarrow{\simeq} Sh^\diamond_{\Lambda_\Sigma}.$$

We now recall the stratification refining $\Lambda_\Sigma$ defined by Bondal-Ruan \cite{Bon06}.
Let $D_i$ for $1\leq i\leq n$ be the standard basis of $\mathbb{Z}^{n}$.  The \newterm{Bondal-Ruan map} is defined by
\begin{align*}
\Phi: \mu_{\mathbb R}^{-1}(0)/M & \rightarrow \widehat{G}\\
\sum_i a_i D_i + M & \mapsto \mu(-\sum_i\floor{a_i}D_i)
\end{align*}
 where $\floor{a_i}$ is the floor of $a_i$. We take the associated stratification $S_{BR}$ on $\mathbb{T}^n$ whose strata are given by the level sets of $S_D := \Phi^{-1}(D)$.  

\begin{lemma} \label{lem: probe restrict}
Let $[p]\in \widetilde M_\mathbb R /M$ denote the coset $p+M$, where $p\in \widetilde M_\mathbb{R}$. The probe sheaf $\mathcal{P}^{\Lambda_{\Sigma}}_{[p]}$ at the point $[p]\in S_D$ is given by
$$
\mathcal{P}^{\Lambda_{\Sigma}}_{[p]}=  i^*\mathcal{P}^{\Lambda_{\widetilde \Sigma}}_D
$$
\end{lemma}
\begin{proof}
First choose $\widetilde{D}$ such that $\mu(\widetilde D)= D$. Then, by Lemma \ref{lem: probe mod M}, $\mathcal{P}_{[p]}^{\Lambda_{\widetilde \Sigma}}=\pi_! \mathcal{P}_p^{\widetilde \Lambda}$. For $p\in S_{\widetilde D}\subset \widetilde{D}+(-1,0]^n$ for some $\widetilde D\in \widetilde M$, write $p=\widetilde D + x$, where $x \in (-1,0]^n$. By Lemma \ref{lem: probe stratification}, we have

\begin{equation} \label{eq: probe} \mathcal{P}_{[p]}^{\Lambda_{\widetilde \Sigma}}=\pi_! \mathcal{P}_p^{\widetilde \Lambda} =\pi_! \mathcal{P}_{\widetilde D}^{\widetilde \Lambda} = \mathcal{P}_D^{\Lambda_{\widetilde \Sigma}}.\end{equation}

\begin{align*}
    \Hom(i^*\mathcal P_D^{\Lambda_{\widetilde \Sigma}}, G)& =\Hom(i^*\mathcal P_D^{\Lambda_{\widetilde \Sigma}}, i^*F) \text{ for some $F$ } & \text{ since $i^*$ is essentially surjective}\\
    & = \Hom(\mathcal P_D^{\Lambda_{\widetilde \Sigma}}, F) & \text{ since $i^*$ is an equivalence }\\
    & = \Hom(\mathcal P_{[p]}^{\Lambda_{\widetilde \Sigma}}, F) & \text{by \eqref{eq: probe}}\\
    & = F_{[p]} & \text{ by the definition of stalk probe}\\
    & = (i^*F)_{[p]} & \text{since $[p]\in S_D\subset \mu^{-1}_\mathbb R(0)/M$} \\
    & = G_{[p]} & \text{by definition of $F$}
\end{align*}
\end{proof}

\begin{lemma} \label{lem: image of kappa}
Let 
\[
\kappa: Coh([U/G]) \to (Sh^w_{\Lambda_\Sigma})^{op}
\]
be the \emph{covariant} equivalence (mirror functor) from \cite{Kuw20}. Then,
$$
\kappa(\bigoplus_{D \in \im \Phi} \mathcal O(D)) =  \mathcal P^{\Lambda_\Sigma}
$$
\end{lemma}
\begin{proof}
    We observe that that each $\kappa_I$ in the mirror functor \[
    \Ind \kappa =\varprojlim_{C(\Sigma)}\Ind \kappa_{\sigma_I}
    \]
    defined by Kuwagaki \cite{Kuw20} takes $\mathcal{O}_{[\mathbb{C}^I\times (\mathbb C^*)^{I^c}/G]}(D)$ to $\mathcal P_D^{\Lambda_{\Sigma_I}|_{M_\mathbb R/M}}=i^* \mathcal P_D^{\Lambda_{\widetilde \Sigma_I}}$. 
Then the \v{C}ech resolution of $\mathcal O_{[U/G]}(D)$ realizes 
\begin{equation} \label{eq: cech}
\mathcal O_{[U/G]}(D) = \varprojlim_{C(\Sigma)}\mathcal{O}_{[\mathbb{C}^I\times (\mathbb C^*)^{I^c}/G]}(D).
\end{equation}
On the other hand for $[p] \in S_D$,
\begin{align}
\mathcal P_{[p]}^{\Lambda_\Sigma} & = i^* \mathcal P_D^{\Lambda_{\widetilde \Sigma}} & \text{ by Lemma~\ref{lem: probe restrict}} \notag\\
& = i^* \pi_! \mathcal P_{\widetilde D}^{\widetilde \Lambda_{I}} & \text{ by Lemma \ref{lem: probe mod M}}\notag\\
& = i^* \pi_!\varinjlim_{C(\Sigma)} \mathcal P_{\widetilde D}^{\widetilde \Lambda_{I}} & \text{ by \cite[Proposition 4.5]{HZ22}} \notag\\
&\  & \text{Note: \cite{HZ22} used opposite indexing i.e. $\sigma_I$ here is $\sigma_{I^c}$ in \cite{HZ22}}\notag \\
& =  i^* \varinjlim_{C(\Sigma)} \pi_! \mathcal P_{\widetilde D}^{\widetilde \Lambda_{I}} & \text{since $\pi_!$ is a left adjoint hence preserves colimit}\notag\\
& = i^* \varinjlim_{C(\Sigma)} \mathcal P_D^{\Lambda_{\widetilde \Sigma_I}} & \text{by Lemma \ref{lem: probe mod M}} \notag\\
& = \varinjlim_{C(\Sigma)} i^* \mathcal P_D^{\Lambda_{\widetilde \Sigma_I}} & \text{since $i^*$ is an equivalence} \label{eq: probe limit}
    \end{align}
In the opposite category, the colimit becomes the limit.

Hence
\begin{align*}
\kappa(
\mathcal O_{[U/G]}(D)) & = \kappa( \varprojlim_{C(\Sigma)}\mathcal{O}_{[\mathbb{C}^I\times (\mathbb C^*)^{I^c}/G]}(D))  & \text{ by \eqref{eq: cech}}\\
& = 
\varinjlim_{C(\Sigma)}\kappa(\mathcal{O}_{[\mathbb{C}^I\times (\mathbb C^*)^{I^c}/G]}(D)) & \text{ since $\kappa$ is an equivalence}\\
& = \varinjlim_{C(\Sigma)} i^*\mathcal P_D^{\Lambda_{\widetilde \Sigma_I}}  & \text{by definition of $\kappa$}\\
& = \mathcal P_{[p]}^{\Lambda_\Sigma} & \text{by \eqref{eq: probe limit}}
\end{align*}

%That is, the \v{C}ech resolution of $\mathcal O(D)$ is term by term mirror to $i^*$ of the explicit description of the probe sheaf $\mathcal P^{\Lambda_{\widetilde \Sigma}}_D=\pi_! \mathcal P^{\widetilde \Lambda}_{\widetilde {D}}$ (by Equation \eqref{eq: probe}), where the formula for $\mathcal P^{\widetilde \Lambda}_{\widetilde {D}}$ can be found in \cite[Proposition 4.5]{HZ22}.

\end{proof}

The following corollary will appear in forthcoming works of Ballard-Duncan-McFaddin \cite{BDM} and Hanlon-Hicks-Lazarev \cite{HHL-b}; both results were publicly claimed before us and use algebro-geometric methods. See also \cite{BC21, HHL-a} for more general discussions on the Rouquier dimension of Fukaya categories and applications via mirror symmetry.

\begin{corollary}\label{cor: toric Frobenius}
    The object $\bigoplus_{D\in \im \Phi} \mathcal{O}(D)$ generates $Coh(X_\Sigma)$.
\end{corollary}

\begin{proof}
Since $\kappa$ is an equivalence \cite{Kuw20}, this follows from Lemmas~\ref{lem: probe mod M} and \ref{lem: image of kappa}.    
\end{proof}

% \begin{remark}
%      Generation of the toric Frobenius pushforward of $\mathcal O$ is proved in \cite{BDM19} in some cases. For a B-side proof of the above statement, we refer to the upcoming work of \cite{BDM}. 
% \end{remark}

As a consequence, we have

\begin{corollary} \label{cor: rdim=dim Cox}
For any toric stack of the form $[U/G]$ in the setup of \cite{Kuw20}, Conjecture~\ref{conj: Orlov} holds i.e.
    $$ \rdim Coh([U/G])=\dim [U/G]$$
    Furthermore, this generation time is achieved by the collection of line bundles $\bigoplus_{D \in \im \Phi} \mathcal O(D)$.
\end{corollary}
\begin{proof}
%We start with the case of $[U/G]$. 
By \cite[Lemma 2.17]{BF12}\footnote{For completeness, we give the reference to the stacky case.  The original proof (for schemes) is due to Rouquier \cite[Proposition 7.17]{Rou08}.}, 
\[
\rdim Coh([U/G])\geq \dim [U/G].
\]

Recall that the Bondal-Ruan stratification $S_{BR}$ is a block stratification by \cite[Corollary 5.7]{FH22}.  By definition, this means it is the coarsening of a CW stratification $S_{CW}$.
Let $S_{CW}^1$ be the first barycentric subdivision of $S_{CW}$.  Then $S_{CW}^1$ is a regular CW stratification. 

 We have a chain of embeddings of closed subskeletons 
$$\Lambda_\Sigma\subset \Lambda_{S_{BR}}\subset \Lambda_{S_{CW}} \subset \Lambda_{S_{CW}^1}.$$
% induces a composition of stop removal functors
% $$ Sh^w_{\Lambda_{S_{CW}}}(\mathbb{T}^{\dim X_{\mathbf{\Sigma}}})\rightarrow  Sh^w_{\bigcup_{S_D\in S_{BR}}ss(i_! \mathbb{C}_{S_D})}(\mathbb{T}^{\dim X_{\mathbf{\Sigma}}})  \rightarrow Sh^w_{\Lambda_\mathbf{\Sigma}}(\mathbb{T}^{\dim X_{\mathbf{\Sigma}}})$$

The upper bound then follows from Theorem~\ref{thm: main theorem} and the non-equivariant coherent-constructible correspondence \cite[Theorem 1.2]{Kuw20} which provides an equivalence
\[
Coh([U/G]) \simeq Sh^w_{\Lambda_\Sigma}.
\]

\

The fact that this is achieved by $\bigoplus_{D \in \im \Phi} \mathcal O(D)$ is Lemma~\ref{lem: image of kappa}.
\end{proof}

We now deduce the result for any normal toric variety $X_{\Sigma'}$, using the result for stacks. 
The following corollary will be translated entirely into toric methods in forthcoming work of Hanlon-Hicks-Lazarev \cite{HHL-b} inspired by a similar approach seen here; their result was claimed publicly before our paper appeared.

\begin{corollary} \label{cor: rdim=dim variety}
For any normal toric variety $X_{\Sigma'}$, 
Conjecture~\ref{conj: Orlov} holds i.e.
    $$ \rdim Coh(X_{\Sigma'})=\dim X_{\Sigma'}.$$
  %  Furthermore, this generation time is achieved by the complex $\bigoplus_{D \in \im \Phi_{\Sigma}} \mathbb R g_*\mathcal O(D)$ where $g: X_\Sigma \to X_{\Sigma'}$ is a resolution of singularities.
\end{corollary}
\begin{proof}
First, observe that we can choose a smooth resolution $X_{\Sigma}$ (see e.g. \cite[Theorem 11.1.9]{CLS}.) Since $X_{\Sigma}$ is smooth, it is isomorphic to the quotient stack $X_{\Sigma}=[U / G]$ where $U$ is obtained from the Cox construction.
% , where
% \[
% \Sigma_{Cox} := \{Cone(e_\rho | \rho \in \sigma) | \sigma \in \Sigma \} \subseteq \mathbb R^{\Sigma(1)}
% \]
Therefore by Corollary~\ref{cor: rdim=dim Cox},
\[
\rdim Coh(X_\Sigma) = \dim X_\Sigma.
\]
Furthermore, since any normal toric variety has rational singularities (see e.g. \cite[Theorem 11.4.2]{CLS}), the functor 
\[
g_*: Coh(X_{\Sigma}) \to Coh(X_{\Sigma'}) 
\]
has dense image (see e.g. \cite[Lemma 7.4]{Kaw19}).  It follows that
\[
\rdim Coh(X_{\Sigma}) \leq \rdim Coh(X_{\Sigma'}) = \dim X_\Sigma' =\dim X_\Sigma
\]
Once again, the lower bound for is \cite[Proposition 7.17]{Rou08}.

\end{proof}

\begin{remark}
To deduce the general case from the smooth case, one may also observe that $\Lambda_{\Sigma'} \subseteq \Lambda_\Sigma$ and then appeal directly to Theorem~\ref{thm: main theorem}.
\end{remark}

\end{document}